\documentclass[12pt]
{amsart}
\usepackage{amsmath,amsthm,amsfonts,amscd,eucal}
\usepackage{mathrsfs}
\usepackage{graphicx} 
\usepackage{epsfig}
\usepackage{typearea}			%Package that allows to control textheigt and width
\areaset{16cm}{23cm}			%determines, using typearea, the height and width

\numberwithin{equation}{section}
\newtheorem{Thm}{Theorem}[section]
\newtheorem{Cor}[Thm]{Corollary}
\newtheorem{Prop}[Thm]{Proposition}
\newtheorem{Lemma}[Thm]{Lemma}

\theoremstyle{definition}
\newtheorem{Dfn}[Thm]{Definition}

\theoremstyle{remark}
\newtheorem{rem}[Thm]{Remark}

\makeatletter
\newenvironment{proofof}[1]{\par
  \pushQED{\qed}%
  \normalfont \topsep6\p@\@plus6\p@\relax
  \trivlist
  \item[\hskip\labelsep
        \bfseries
    Proof of #1\@addpunct{.}]\ignorespaces
}{%
  \popQED\endtrivlist\@endpefalse
}
\makeatother

\def\ca{{\mathcal A}}
\def\cb{{\mathcal B}}
\def\cc{{\mathcal C}}
\def\cd{{\mathcal D}}
\def\ce{{\mathcal E}}

\def\ch{{\mathcal H}}

\def\ck{{\mathcal K}}
\def\cl{{\mathcal L}}

\def\cp{{\mathcal P}}

\def\ct{{\mathcal T}}

\def\cx{{\mathcal X}}

\def\cz{{\mathcal Z}}

\def\bc{{\mathbb C}}

\def\bn{{\mathbb N}}
\def\br{{\mathbb R}}

\def\a{\alpha}

\def\g{\gamma}        
\def\d{\delta}        
\def\eps{\varepsilon}

\def\th{\vartheta}

\def\l{\lambda}       
\def\m{\mu}

\def\r{\rho}
\def\s{\sigma}        
     
\def\t{\tau}

\def\f{\varphi}

        \def\O{\Omega}

\def\wh{\widehat}

\DeclareMathOperator{\Tr}{Tr}
\DeclareMathOperator{\tr}{Tr}

\DeclareMathOperator{\re}{Re}
\DeclareMathOperator{\Res}{Res}
\newcommand{\norm}[1]{\| #1 \|}

\newcommand{\dg}{d_{\text{geo}}}
\newcommand{\Less}{L_{\text{ess}}}
\newcommand{\dess}{d_{\text{ess}}}

\begin{document}

\title{Spectral triples for nested fractals}
%Autori
\author{Daniele Guido, Tommaso Isola}
\address{Dipartimento di Matematica, Universit\`a di Roma ``Tor
Vergata'', I--00133 Roma, Italy.}
\email{guido@mat.uniroma2.it, isola@mat.uniroma2.it}
\thanks{The  authors were partially 
supported by GNAMPA, MIUR,  GDRE GREFI GENCO, and the ERC Advanced Grant 227458 OACFT }

\begin{abstract}
It is shown that, for nested fractals \cite{Lind}, the main structural data, such as the Hausdorff dimension and measure, the geodesic distance (when it exists) induced by the immersion in $\br^n$, and the self-similar energy can all be recovered by the description of the fractals in terms of the spectral triples considered in \cite{GuIs10}. %We describe in particular the case of the Vicsek square, showing that all self-similar energies can be described through a deformation of the square to a rhombus.
\end{abstract}
  
 \subjclass[2000]{Primary 58B34; Secondary 28A80, 58J42}
\keywords{spectral triple, nested fractal, self-similar energy, Hausdorff dimension, noncommutative distance}

\date{Jan 29, 2016}

\maketitle

 \setcounter{section}{-1}

\section{Introduction}

In this note, we analyze the class of nested fractals \cite{Lind,Sabot,Pe00,Pe04} by
making use of the spectral triples introduced in \cite{ConnesBook} for the case of the Cantor set, and in \cite{GuIs10} for a wide class of fractals, and prove that such noncommutative geometric description reproduces the classical notions of Hausdorff dimension and measure, the self-similar Dirichlet form, and also, when the fractal is finitely arcwise connected \cite{K}, the corresponding geodesic distance.

Starting with the first examples given by Connes in \cite{ConnesBook}  and the early papers of Lapidus \cite{La94,La97}, many papers are now available concerning the noncommutative approach to fractals \cite{Buyalo,GuIs8,GuIs9,GuIs10,ChIv07,PeBe,CIL,CIS,FaSa,CGIS02,LaSa}. It turns out that noncommutative geometry
%, and in particular spectral triples, 
can be fruitfully applied to smooth as well as singular spaces, since it  gives a universal procedure which associates with a spectral triple a metric dimension, an integration, a distance and an energy.
% many of the features of self-similar fractals may be recovered in this way.
%Indeed,  the idea pursued in the above mentioned papers was to construct, starting from geometrical data of the fractal, a spectral triple able to  reproduce typical notions of the analysis on fractals such as dimension, measure, distance and energy. 
In all of  the above mentioned papers notions of noncommutative dimension and/or measure were studied, while few of them \cite{GuIs9,PeBe,CIL,CIS,LaSa} could fully recover a natural distance on some class of fractals. A noncommutative construction of the Dirichlet energy for fractals starting from geometric data was only considered in \cite{CGIS02} for the case of the Sierpinski gasket. 
%The noncommutative approach was extended from fractals to general metric spaces in \cite{Palm}, where spectral triples were constructed that can recover the Hausdorff dimension and measure, and the metric. 
The idea of constructing a spectral triple on a fractal as a countable direct sum of finite dimensional spectral triples has been  extended from fractals to general compact metric spaces in \cite{Palm}, where  the Hausdorff dimension and measure, and the metric were recovered from their noncommutative analogues. More recently, in \cite{HKT} a different approach to constructing spectral triples on metric spaces was taken, based on \cite{CiSa1,CiSa2}, so the starting point is a regular symmetric Dirichlet form on a locally compact separable metric space, endowed with a nonnegative Radon measure, and the intrinsic (or Carnot-Carath\'eodory) metric is recovered.

%Palm,HKT,JKS

The basic requirements for a spectral triple $\ct=(\ca,\ch,D)$, where $\ca$ is a self-adjoint algebra of operators and $D$ is an unbounded self-adjoint operator, both acting on the Hilbert space $\ch$, are the boundedness of the commutators $[D,a]$, $a\in\ca$, and the compactness of the resolvents of $D$  \cite{ConnesBook}. Based on these hypotheses, one may associate with a spectral triple $\ct=(\ca,\ch,D)$ a notion of dimension, an integral on the elements of the algebra,  an energy form, and a distance on the state space of the algebra, according to the table below, where  we denote by $\mathcal{Z}_a(s)=\tr(a(D^2+1)^{-s/2})$ the zeta function for $a\in\ca$, by $\cz(s)=\cz_I(s)$ the zeta function of the spectral triple,  and $\f,\psi $ are states on the norm closure of $\ca$:
\begin{table}[h]
%\caption{default}
\begin{center}
\begin{tabular}{|l|l|}
$d_H=\inf\{s>0:\mathcal{Z}(s)<+\infty\}$& (NC dimension)\\
\hline
%\\
$\oint a=\displaystyle\Res_{s=d_H} \mathcal{Z}_a(s), a\in\ca$ & (NC integral)\\
\hline
%\\
$d_D(\f,\psi)=\displaystyle\sup_{\|[D,a]\|\leq1}|\f(a)-\psi(a)|$ & (NC distance)\\
\hline
$\ce_D[a]=\displaystyle\oint|[D,a]|^2,  a\in\ca$ &  (NC Dirichlet energy)
\end{tabular}
\end{center}
\label{default}
\end{table}%

The formula for the noncommutative energy  is motivated by the fact that, in noncommutative geometry,  $[D,a]$ is a replacement for the gradient of $a$. Similar expressions were used in some previous papers \cite{PeBe, JuSa,CGIS02}.
%\begin{align*}
%d_H=\inf\{s>0:Z_I(s)<+\infty\}&\text{ (NC dimension)}\\
%%\vol(\ct)=\Res_{s=d_H}\tr(D^2+1)^{-s/2}&\text{ (Hausdorff NC volume)}\\
%\oint a=\Res_{s=d_H}Z_a(s), a\in\ca &\text{ (NC integral)}\\
%\ce_D[a]=\oint|[D,a]|^2, &\text{ (NC Dirichlet energy)}.\\
%\text{ If } \f,\psi \text{ are states on the closure of } \ca,&\\
%d_D(\f,\psi)=\sup_{\|[D,a]\|\leq1}{|\f(a)-\psi(a)|}, &\text{ (NC distance)}\\
%\end{align*}

%As is known, see references above, there are many possible spectral triples that can be associated to fractals, and more generally to singular spaces. Our aim here is to show that with a very simple triple, based only on few information concerning the self-similar fractal, many features which are usually produced with clever analytic tools can be completely recovered, for nested fractals, via general noncommutative methods, based on completely different ideas, such as that of noncommutative residues and singular traces.
As the references above show, there are many possible spectral triples that can be associated to a fractal, and more generally to a singular space. Our aim here is to show that with a very simple triple, based only on a few data from the self-similar fractal, many features which are usually produced with clever analytic tools can be recovered, for nested fractals, via general noncommutative methods, based on completely different ideas, such as those of noncommutative residues and singular traces.

The  spectral triple used here was introduced in \cite{GuIs10} and consists of the algebra $\ca$ of suitably regular functions on the fractal acting on the Hilbert space  $\ch$ given by the $\ell^2$ space on the oriented edges of the fractal, and of the Dirac operator $D$ on $\ch$ which maps an oriented edge to its opposite, multiplied by the inverse length of the edge itself. We call these triples discrete since the Hilbert space is not given by an $L^2$ space w.r.t. some measure on the fractal but the $\ell^2$ space of its edges, see Definition \ref{DiscreteTriples} for further details.

In Section 2 we analyze the noncommutative metric and dimension associated with our triple. This was already done in \cite{GuIs10} for a much larger family of fractals, the results for the self-similar case being stated in Remark 2.11 $(ii)$ without  proof.
We prove here that, for nested fractals, the noncommutative dimension coincides with the Hausdorff dimension and the noncommutative integral coincides (up to a multiplicative constant) with the integral w.r.t.~the Hausdorff measure.
%We recall here that the noncommutative dimension $d$ is given by abscissa of convergence of the zeta function $z(s)=\tr(|D|^{-s})$, while the integral of $f$ is given by the residue in $d$ of the zeta function $z_f(s)=\tr(f|D|^{-s})$.

In Section 3 we study the noncommutative Dirichlet energy. In \cite{CGIS02}  the case of the Sierpinski gasket was considered, with a spectral triple given by a deformation of that considered in \cite{CIL,CIS}. The noncommutative Dirichlet energy was introduced there, in terms of the residue of a zeta function,   and it was proved that it coincides with the unique self-similar energy on the gasket up to a constant. However, in order to obtain a non trivial energy, the abscissa of convergence of the zeta function has to take a specific value $\d$, which we called energy dimension, and is different from the Hausdorff dimension $d_H$.
Here we show that this result holds for all nested fractals endowed with discrete spectral triples, and prove a general formula for  the energy dimension, namely $\d=2-\frac{\log \rho}{\log \lambda}$, 
%for finite energy functions,
where $\rho$ is the scaling factor for the energy, and $\lambda$ is the (unique) scaling factor for the contractions.

The formula for the self-similar energy proved here for nested fractals is based on a clever result of Peirone \cite{Pe04}, where he shows that any quadratic functional on the space of functions on the essential fixed points gives rise to a self-similar Dirichlet form on the whole fractal, and that the set of energies obtained in this way coincides with those obtained from eigenforms. This result, together with the simple form of our triples, allows a quite short  proof of the formula for the self-similar energy. However,
% what is remarkable in our opinion is the validity of the noncommutative energy formula, which was proposed in [CGIS02] for the Sierpinski gasket, for the whole class of nested fractals equipped with our discrete spectral triples. Moreover, 
even though the noncommutative energy formula was envisaged as a generalisation of that for the noncommutative integral, in the case of fractals the reason why it holds is technically different from the reason why the noncommutative integral formula holds. The latter is based on the summability of continuous functions w.r.t. the Hausdorff measure, while square commutators $|[D,f]|^2$ are not summable in general w.r.t. the noncommutative Hausdorff trace for a finite energy function $f$. Indeed, in order to obtain a finite residue, one should change the dimension, passing from the Hausdorff dimension to the energy dimension.

Moreover, in the case of (possibly noncommutative) smooth manifolds, the square of the Dirac operator is (related to) the Laplace operator, and the Hilbert space is the $L^2$-space of a (finite dimensional) fiber bundle, namely is a finitely generated projective module. This implies that  the square commutator $|[D,f]|^2$ is essentially bounded, when $f$ is Lipschitz, and its noncommutative Hausdorff trace (w.r.t. the standard dimension) gives the classical Dirichlet energy. In the case of fractals instead the intersection between Lipschitz functions and finite energy functions may be quite small, and for our triples the Hilbert space is not finitely generated projective. Therefore the summability of the function $|D|^{-s/2}|[D,f]|^2|D|^{-s/2}$ for a finite energy $f$ and sufficiently large $s$ holds for reasons completely different from those responsible of the summability of   $|D|^{-s/2}f|D|^{-s/2}$ for a Lipschitz $f$, and the abscissae of convergence are different.

We remark here that our triple is uniquely associated with the fractal as a self-similar metric space, hence it produces a unique energy for any metric fractal. Then one may ask what happens for the Vicsek square, which admits infinitely many. It turns out that any of such energies can be produced, with our method, by suitably deforming the fractal. In other words, we prove, in Section 3.2, that  any of the self-similar energies on the Vicsek square comes, via our technique, from a different immersion in $\br^2$, related to the original square by an affine transformation; the energy is driven by the geometry.

In Section 4 we prove that the geodesic distance induced by the natural immersion of the nested fractal in $\br^n$ (when finite) can be completely recovered by means of the discrete spectral triple. Let us remark that, since the algebra of functions on the fractal is abelian, we may consider as states the delta-functions, thus obtaining a distance between points of the fractal. However the formula for the noncommutative distance should be modified, by replacing the norm $\|[D,a]\|$ with the essential norm. We remark that such formula for the distance involving the essential norm is, to our knowledge, completely new, and seems to be related to the approximation of the fractal via finite graphs.

In some cases, such as the Sierpinski gasket, the replacement of the norm $\|[D,a]\|$ with the essential norm is not necessary, as  stated without proof in \cite{GuIs10}. An analogous result for the gasket, but with a completely different spectral triple, was proved in \cite{CIL,CIS}. 

The results in this paper have been announced in the conference ``Noncommutative Analysis, Operator Theory, and Applications", held in Milan on June 2014 \cite{GuIsProc}, and in the Special Session ``Fractals" of the  10$^{th}$ AIMS Conference on Dynamical Systems,
Differential Equations and Applications, held in Madrid on July 2014.

\section{Preliminaries}
\subsection{Nested fractals}
Let $\O:= \{w_i : i=1,\ldots,k \}$ be a family of contracting similarities of $\br^{N}$, $i.e.$ there are $\l_i \in(0,1)$ such that $\| w_i(x) - w_i(y) \| = \l_i \|x-y\|$, $x,y\in \br^{N}$.  	 The unique non-empty compact subset $K$ of $\br^{N}$ such that $K = \bigcup_{i=1}^{k} w_i(K)$  is called the {\it self-similar fractal} defined by $\{w_i  \}_{i=1,\ldots,k}$. For any $i\in\{1,\ldots,k\}$, let $p_i\in\br^N$ be the unique fixed-point of $w_i$, and say that $p_i$ is an essential fixed-point of  $\O$ if there are $i',j,j'\in\{1,\ldots,k\}$ such that $i'\neq i$, and $w_j(p_i)=w_{j'}(p_{i'})$. Denote by $V_0$ the set of essential fixed-points of $\O$, and we assume that it has at least two elements, and let $E_0:=\{ (p,q) : p,q\in V_0, p\neq q\}$. Observe that $(V_0,E_0)$ is a directed finite graph \cite{BoMu} whose edges are in 1:1 correspondence with ordered pairs of distinct vertices.
For any $n\in\bn$, set $\Sigma_n := \{ \s: \{1,\ldots,n\} \to \{1,\ldots,k\} \}$, $w_\s:=w_{\s(1)}\circ\cdots\circ w_{\s(n)}$, $\forall \s\in\Sigma_n$, $V_n := \cup_{\s\in\Sigma_n} w_\s(V_0)$, and $w_\emptyset:=id$, $\Sigma_0:=\{\emptyset\}$, $\Sigma :=\cup_{n=0}^\infty \Sigma_n$. Then, $V_{n-1}\subset V_{n}$, $\forall n\in\bn$. Sets of the form $w_\s(V_0)$, for $\s\in\Sigma_n$, are called \textit{combinatorial $n$-cells}, while those of the form $w_\s(K)$ are called \textit{$n$-cells}. For any $n\in\bn$, define $E_n :=\{ (w_\s(p),w_\s(q)) : \s\in\Sigma_n, p,q\in V_0, p\neq q\}$, and, 
for any $\s\in\Sigma_n$, $i\in\{1,\ldots,k\}$, denote by $\s\cdot i\in\Sigma_{n+1}$ the map defined by $\s\cdot i(j)=\s(j)$, $j\in\{1,\ldots,n\}$, $\s\cdot i(n+1)=i$.

\begin{Dfn} \label{def:nested}
The couple $(K,\O)$ is said to be a \textit{nested fractal} in the sense of Lindstr\o m  \cite{Lind,Barl} if
\begin{itemize}
\item[$(1)$] $\l_i=\l$, for all $i\in\{1,\ldots,k\}$,

\item[$(2)$] there is an open bounded set $U\subset\br^N$, such that $\cup_{i=1}^k w_i(U)\subset U$, and $w_i(U)\cap w_j(U)=\emptyset$, for all $i,j\in\{1,\ldots,k\}$, $i\neq j$ (open set condition),

\item[$(3)$] the graph $(V_1,E_1)$ is connected, that is, for any $p,q\in V_1$, there are $p_0,\ldots,p_s\in V_1$, such that $p_0=p$, $p_s=q$, and $(p_{i-1},p_{i})\in E_1$, for all $i=1,\ldots,s$,

\item[$(4)$] if $\s,\s'\in\Sigma_n$, $\s\neq \s'$, then $w_\s(V_0)\neq w_{\s'}(V_0)$, and $w_\s(K)\cap w_{\s'}(K) = w_\s(V_0)\cap w_{\s'}(V_0)$ (nesting property),

\item[$(5)$] if $p,q\in V_0$, $p\neq q$, then the symmetry with respect to $\Pi_{pq}:=\{ z\in\br^N : \norm{z-p}=\norm{z-q}\}$ maps combinatorial $n$-cells to combinatorial $n$-cells, for any $n\in\bn\cup\{0\}$, and maps a combinatorial $n$-cell lying on both sides of $\Pi_{pq}$ to itself (symmetry property).
\end{itemize}

\end{Dfn}
 
\begin{rem} 
Any (open) regular polygon of the plane can be taken as the set $U$ in Definition \ref{def:nested} for some nested fractal. In \cite{Barl,Kiga2} it is conjectured that $V_0$ can only be the vertex set of a regular planar polygon, or of an $N$-dimensional tetrahedron, or of the convex envelope of the set $\{e_i,-e_i : i=1,\ldots,N\}$, where $\{e_1,\ldots,e_N\}$ is the canonical basis of $\br^N$.
\end{rem}
 
Set $V:=\cup_{n=0}^\infty V_n$, $E:=\cup_{n=0}^\infty E_n$, and, for any $e=(e^+,e^-)\in E$, and any function $f$, set $\langle \partial f, e \rangle := f(e^+)-f(e^-)$. For any $\s\in \Sigma$, set $|\s|=n$, if $\s\in\Sigma_n$.

The following definitions are taken from \cite{Pe00,Pe04}, with slight modifications.

\begin{Dfn} \label{FS05}
If $X$ is a finite set, let us denote by $C(X)$ the set of functions from $X$ to $\br$, and by $\mathscr{D}$ the set of quadratic functionals $\ce: C(V_0)\to \br$ of the form
$$
\ce[f]:=\sum_{(p,q)\in E_0} c_{pq}(f(p)-f(q))^2 = \sum_{e\in E_0} c_e \, | \langle \partial f, e \rangle |^2,
$$ 
where $c_{pq}=c_{qp}>0$, $(p,q)\in E_0$. For $\ce\in\mathscr{D}$, $n\in\bn$,  set 
$$
	S_n(\ce)[f]  := \sum_{\s\in\Sigma_n} \ce[f\circ w_\s], \qquad \forall f\in C(V_n).
$$
	A functional $\ce\in\mathscr{D}$ is called an eigenform, with eigenvalue $\r>0$, if $\inf \{ S_1(\ce)[g] : g\in C(V_1), g|_{V_0}=f\}=\r\ce[f]$, $\forall f\in C(V_0)$.
\end{Dfn}

Lindstr\o m proved that there is an eigenform $\wh{\ce}\in\mathscr{D}$. Note that all eigenforms have the same eigenvalue $\r$, which satisfies $\r\in(0,1)$, see \cite{Sabot}, Proposition 3.8. It is known that $\wh{\ce}_\infty[f]:= \lim_{n\to\infty} \r^{-n}S_n(\wh{\ce})[f]$ defines a Dirichlet form on the fractal $K$. Define $\mathscr{F}:= \{ f\in \cc(K) : \wh{\ce}_\infty[f]<\infty\}$. It turns out that $\mathscr{F}$ does not depend on the eigenform, cf. \cite{Pe04} Section 3.

\begin{Thm} \label{FS06}
	Let $\ce\in\mathscr{D}$, not necessarily an eigenform. Then there exists 
	$$
	\ce_\infty[f]:=\lim_{n\to\infty} \r^{-n}S_n(\ce)[f], \qquad f\in\mathscr{F}.
	$$
	 Moreover, there is an eigenform $\ce'\in\mathscr{D}$ such that $\ce_\infty = \ce'_\infty$.
\end{Thm}
\begin{proof}
See \cite{Pe04}, Theorem 4.11, and Remark 4.1.
\end{proof}

Therefore, one can obtain the Dirichlet form on the fractal $K$ just by taking any quadratic form $\ce\in\cd$, and computing $\ce_\infty$. 

\begin{rem}
In the general case of p.c.f. self-similar fractals, the self-similar Dirichlet forms were introduced in \cite{Kiga1,Kusu}, see also \cite{Kiga2} where one can find the previous results, and much more. Some results in the setting of noncommutive geometry are contained in the pioneering papers \cite{KiLa1,KiLa2}.
\end{rem}

\subsection{Spectral triples}

 Let us recall that $(\ca,\ch,D)$ is called a {\it spectral triple} when  $\ca$ is a $^*$-algebra acting on the Hilbert space $\ch$, $D$ is a self  adjoint operator on the same Hilbert space such that $[D,a]$ is  bounded for any $a\in\ca$, and $D$ has compact resolvent.  In the  following we shall assume that $0$ is not an eigenvalue of $D$. In the general case one should replace $|D|^{-d}$ below with e.g. $(I+D^2)^{-d/2}$. Such a triple is called $d^+$-summable, $d\in  (0,\infty)$, when $|D|^{-d}$ belongs to the Macaev ideal  $\cl^{1,\infty}=\{a:\frac{S_{n}(a)}{\log n}<\infty\}$, where $S_n(a):=\sum_{k=1}^n \m_k(a)$ is the sum of the first $n$ largest eigenvalues (counted with multiplicity) of $|a|$.  
 
For $d^+$-summable spectral triples, the zeta function of the spectral triple is defined as $\cz(s)=\tr(|D|^{-s})$, and the (noncommutative) metric dimension $d$ of  $(\ca,\ch,D)$ is defined as the abscissa of convergence of $\cz(s)$. The dimension spectrum is the larger set consisting of the poles of the meromorphic extension of $\cz(s)$.
 The noncommutative version of the integral on functions is given by  the formula $\oint a=\Tr_\omega(a|D|^{-d})$, where, for $T>0$, $\Tr_\omega(T):= \lim_\omega \frac{S_{n}(T)}{\log n}$ is the  Dixmier trace, i.e. a singular trace summing logarithmic divergences.   
 It is well known (cf. \cite{ConnesBook} and \cite{CPS} Thm. 3.8) that the noncommutative integral can be computed as
 $$
 \oint a=\tr_\omega (a|D|^{-d}) = \frac1d{\Res}_{s=d}\ \tr(a|D|^{-s})
 $$
when the limit in the definition of residue exists. 

In analogy with the classical setting, one may define the  $\a$-dimensional Hausdorff functional as the map $a\mapsto\oint a= \Tr_\omega(a |D|^{-\a})$, and it turns out  that such functional on $\ca$ can be non-trivial only if $\a=d(\ca,\ch,D)$  \cite{GuIs9}.
However, when considering the Dirichlet form as in the introduction, things may change. Given a spectral triple, one may define an associated energy functional as 
$$
\ce[a]={\Res}_{s=d}\ \tr(|[D,a]|^2|D|^{-s}).
$$
Indeed, in the case of manifolds, this formula gives back the Dirichlet integral. For fractal spaces instead, it turns out that the correct Hausdorff functional is not related to the metric dimension of the spectral triple, but to another dimension, cf. \cite{CGIS02} and Section 3 below.

We conclude this brief description of some features of spectral triples with the notion of noncommutative distance. Given two states $\f,\psi$ on the C$^*$-algebra generated by $\ca$, their distance is defined as
$$
d_D(\f,\psi)=\sup_{\|[D,a]\|\leq1}|\f(a)-\psi(a)|.
$$
Under suitable circumstances, such distance metrizes the weak$^*$-topology on states (cf. \cite{Rieffel99}).
In the commutative case, one may choose delta-functions as states, thus recovering a distance on points. In Section 4 we propose a modification of this formula in order to obtain the original distance on nested fractals.

\section{Spectral triples on self-similar fractals}\label{triplesonfractals}

A standard way to construct spectral triples on  a self-similar fractal $K$ is the following:

\begin{itemize}
\item Select a subset $S\subset K$ together with a triple $\ct_o=(\pi_o,\ch_o,D_o)$ on $\cc(S)$.
\item Set  $\ct_\emptyset=(\pi_\emptyset,\ch_\emptyset,D_\emptyset)$ on $\cc(K)$, where $\pi_\emptyset(f)=\pi_o(f|_S)$, $\ch_\emptyset =\ch_o$, $D_\emptyset=D_o$.
\item Set  $\ct_\sigma:=(\pi_\sigma,\ch_\emptyset,D_\sigma)$ on $\cc(K)$, with
$\pi_\sigma(f)=\pi_\emptyset(f\circ w_\sigma)$, $D_\sigma=\lambda_\sigma^{-1}D_\emptyset$, $\lambda_\sigma=\prod_{i=1}^{|\sigma|}\l_{\sigma(i)}$.
\item Set $\ct=\bigoplus_\sigma\ct_\sigma$ on $\cc(K)$. The $^*$-algebra $\ca$ is usually taken as $\ca=\{f\in\cc(K):[D,f]$ is bdd\}, if not otherwise stated.
\end{itemize}

\begin{Dfn}\label{DiscreteTriples}[Discrete triple on self-similar fractals]
Assume $K$ to be a self-similar fractal in $\br^n$, and construct a triple $\ct_o=(\pi_o,\ch_o,D_o)$ on $\cc(V_0)$, $D_o=F_o|D_o|$ the polar decomposition of the Dirac operator, as follows:  
%for any $e=(e^+,e^-)\in E_0$, set $\partial e=\{e^+,e^-\}$,  let $\ell(e)>0$ be its length, and set
%$$
%\ch_o=\bigoplus_{e\in E_0}\ell^2(\partial e),\quad\pi(f) =\bigoplus_{e\in E_0}
%\begin{pmatrix}
%f(e^+)&0\\0&f(e^-)
%\end{pmatrix}
% \quad D_o=\bigoplus_{e\in E_0}\frac1{\ell(e)}
%\begin{pmatrix}
%0&1\\1&0
%\end{pmatrix}
%$$
%or
$$
\ch_o = \bigoplus_{e\in E_0} \ell^2(e),
\quad \pi_o(f) = \bigoplus_{e\in E_0} f(e^+) 
\quad |D_o| = \bigoplus_{e\in E_0}\ell(e)^{-1}I
$$
where $\ell(e)>0$ denotes the length of an edge and $F_o$ is the self-adjoint unitary sending an oriented edge to the same edge with the opposite orientation.
Then construct the triples $\ct_\sigma$ and $\ct=\bigoplus_\sigma\ct_\sigma$ as above.
\end{Dfn}
%
%Setting $E_0=V_0\times V_0$ we have already assumed that any two vertices of the same cell are connected by an edge. We also assume that the elements of $V_0$ are fixed points, so that $V_n\subset V_{n+1}$.

\begin{rem}
$(1)$ Observe that the spectral triple $\ct_o$ fully decomposes as a direct sum on unoriented edges, where, for any unoriented edge $e$, the Hilbert space is a copy of $\bc^2$, $\pi_o(f)$ acts as the diagonal matrix consisting of the values of $f$ at the end points of $e$, and $D_o =\ell(e)^{-1} \begin{pmatrix}
0 & 1 \\
1 & 0
\end{pmatrix}$.

%The usual definition of $D_o$ is $D_o = \oplus_{\eps \in E_g} \frac1{\ell(\eps)} \begin{pmatrix}
%0 & 1 \\
%1 & 0
%\end{pmatrix}$, where $E_g$ is the set of geometric edges, that is any $\eps\in E_g$ coincides with a pair of oppositely oriented edges in $E_0$, and $\ell(\eps)$ is its length. It is obvious that the two definitions coincide.
%
\noindent $(2)$ 
Discrete spectral triples are characterized by the fact that the subspace $S$ on which the spectral triple is based is 
%treated as 
a discrete space, or, the corresponding Hilbert space is finite dimensional. For the Sierpinski gasket, spectral triples for which the subspace $S$ is 
%treated as 
a continuous space have been considered, e.g. in
\cite{CIL,CIS,CGIS02}, where the subset $S$ is homeomorphic to a circle.
\end{rem}

\begin{Thm}\label{thm:volmeas}
Assume $K$ to be a nested fractal in $\br^n$. The zeta function $\cz(s)=\tr(|D|^{-s})$ of $(\ca,\ch,D)$ has a meromorphic extension  given by
\[
\mathcal{Z} (s)=\frac{ \sum_{e\in E_0} \ell(e)^s}{1-k\l^{ s}}\, .
\]
Therefore, the metric dimension $d_D$ of the spectral triple $(\ca,\ch,D)$ is $d_D=d=\frac{\log k}{\log 1/\l}$ and
the dimensional spectrum  of the spectral triple is
\[
\mathcal{S}_{\it dim} = \{ d\left(1+\frac{2\pi i}{\log k}n \right): n\in \mathbb{Z}\} \subset \mathbb{C}\, .
\]
$\mathcal{Z}$ has a simple pole in $d_D$, and the measure associated via Riesz theorem with the functional $f\to\oint f$ coincides with a multiple of the Hausdorff measure $H_d$ (normalized on $K$):
\[
 \oint f=tr_\omega(f|D|^{-d}) = \frac{1}{ \log k} \sum_{e\in E_0} \ell(e)^d \int_K f\, dH_d \qquad f\in C(K).
\]
\end{Thm}
\begin{proof}
The eigenvalues of  $|D_\s|$ are exactly $\{ \frac{ 1 }{ \ell(e) \l^{|\s|} } \}_{e\in E_0}$, each one with multiplicity $1$.

Hence $\tr(|D_\s|^{-s}) =  \l^{s|\s|} \sum_{e\in E_0} \ell(e)^s$ and for $\re s>d$ we have
\begin{align*}
\tr(|D|^{-s})
&=\sum_\s \tr(|D_\s|^{-s})= \sum_{e\in E_0} \ell(e)^s  \sum_\s  \l^{s|\s|} 
 = \sum_{e\in E_0} \ell(e)^s \sum_{n\geq0} \sum_{|\s|=n} \l^{sn} \\
& = \sum_{e\in E_0} \ell(e)^s \sum_{n\geq0} k^{n} \l^{sn}
= \sum_{e\in E_0} \ell(e)^s (1-k\l^s)^{-1}\, .
\end{align*}
Therefore, we have $\mathcal{S}_{\it dim} = \{d\left(1+\frac{2\pi i}{\log k}n \right): n\in \mathbb{Z}\}\subset \mathbb{C}$.
Now we prove that the volume measure is a multiple of the Hausdorff measure $H_d$.
Clearly, the functional $\tr_\omega(f|D|^{-d})$ makes sense also for bounded Borel functions on $K$, and we recall that the logarithmic Dixmier trace may be calculated as a residue (cf. \cite{ConnesBook} and \cite{CPS} Thm. 3.8): $d\tr_\omega(f|D|^{-d})={\Res}_{s=d}\ \tr(f|D|^{-s})$, when the latter exists. Then, for any multi-index $\t$, denoting by $C_\t:= w_\t(K)$ a cell of $K$, we get
\begin{align*}
d\tr_\omega (\chi_{C_\t}|D|^{-d})
&={\Res}_{s=d}\ \tr(\chi_{C_\t}|D|^{-s})
=\lim_{s\to d^+}(s-d)\ \tr(\chi_{C_\t}|D|^{-s})\\
&=\lim_{s\to d^+}(s-d) \sum_\s \tr(\chi_{C_\t}\circ w_\s|D_\s|^{-s}),
\end{align*}
and we note that $\chi_{C_\t}\circ w_\s$ is not zero either when $\s<\t$ or when $\s\geq \t$. In the latter case, $\chi_{C_\t}\circ w_\s=1$.
Observe that $\tr(\chi_{C_\t}\circ w_\s |D_\s|^{-s})\leq \tr(|D_\s|^{-s}) =  \l^{s|\s|} \sum_{e\in E_0} \ell(e)^s \to  \l^{d|\s|} \sum_{e\in E_0} \ell(e)^d$ when $s\to d^+$, hence
 $\displaystyle\lim_{s\to d^+}(s-d)\tr(\chi_{C_\t}\circ w_\s  |D_\s|^{-s})=0$. Therefore we may forget about the finitely many $\s<\t$, and get
\begin{align*}
d\tr_\omega (\chi_{C_\t}|D|^{-d})
&=\lim_{s\to d^+} (s-d)\sum_{\s\geq\t} \tr(|D_\s|^{-s})
=\lim_{s\to d^+} (s-d) \sum_{n=0}^\infty k^n \l^{s(|\t|+n)} \sum_{e\in E_0} \ell(e)^s  \\
&=  \l^{d|\t|} \sum_{e\in E_0} \ell(e)^d  \lim_{s\to d^+} \frac{s-d}{1-k \l^{s }}
= \frac{1}{k^{|\t|} \log1/\l} \sum_{e\in E_0} \ell(e)^d\\
&= \frac{1}{ \log1/\l} \sum_{e\in E_0} \ell(e)^d H_d(C_\t)\, .
\end{align*}
This implies that, for any $f\in\cc(K)$ for which $f\leq\chi_{C_\t}$, $\oint f\leq \frac{1}{ \log k} \sum_{e\in E_0} \ell(e)^d \left(\frac1k\right)^{|\t|}$, therefore points have zero volume, and $\oint \chi_{\dot{C}_\t}=\oint\chi_{C_\t}$, where $\dot{C}_\t$ denotes the interior of $C_\t$.
As a consequence, for the simple functions given by finite linear combinations of characteristic  functions of cells or vertices, $\oint\f = \frac{1}{ \log k} \sum_{e\in E_0} \ell(e)^d \int\f\,d H_d$.
Since continuous functions are Riemann integrable w.r.t. such simple functions, the thesis follows.
\end{proof}

\begin{rem} \label{notnested}
The above proof holds in greater generality than stated, since it doesn't use the symmetry property of the nested fractal $K$.
\end{rem}
 
\section{A noncommutative formula for the Dirichlet energy}

As explained in the introduction, we propose the following expression for  the energy form on a spectral triple :
$$
\Res_{s=\delta}\tr (|D|^{-s/2}|[D,f]|^2\, |D|^{-s/2}).
$$
However, while for smooth manifolds $\d$ coincides with the dimension, for singular structures such as fractals the metric dimension $d$ is in general different from the energy dimension $\delta$ \cite{CGIS02}. Here we characterize the energy dimension for nested fractals in terms of the scaling parameters for the distance and the energy, recovering in particular the same value as in \cite{CGIS02} for the Sierpinski gasket. We also note that for fractals, elements $a$ with finite energy are not necessarily Lipschitz, namely $[D,a]$ is unbounded in general, however $[D,a]|D|^{-s/2}$ is Hilbert-Schmidt, for  $s>\d$. A feature of the discrete spectral triples we are using is that the operators $|D|$ and $|[D,f]|$ are diagonal w.r.t. the basis of the Hilbert space made of oriented edges, in particular they commute, so we may replace  $|D|^{-s/2}|[D,f]|^2\, |D|^{-s/2}$ with $|[D,f]|^2\, |D|^{-s}$, and some computations are greatly simplified.
As in  \cite{CGIS02}, the residue formula seems to be more efficient than  the Dixmier trace formula, therefore we only discuss the former. 
%
%Our main theorem here is the following.

\subsection{The residue formula}
\begin{Thm}\label{thm:energy}
Let $K$ be a nested fractal with scaling parameter $\lambda$, eigenvalue $\rho$ for the eigenform, and $\mathscr{F}$ the set of finite energy functions, endowed with the discrete spectral triple described above. Then, for any non-constant $f\in\mathscr{F}$, the abscissa of convergence of $Z_{|[D,f]|^2}(s)=\tr (|[D,f]|^2\, |D|^{-s})$ is equal to $\delta=2-\frac{\log \rho}{\log \lambda}$, and the residue produces a self-similar Dirichlet form:
$$\ce_D[f]= \Res_{s=\delta}Z_{|[D,f]|^2}(s)=(\log1/\lambda)^{-1}\ce_\infty[f],$$ 
where $\ce$ on $C(V_0)$ is given by $
\ce[f]=\sum_{e\in E_0}\ell(e)^{\d-2} |\langle \partial f,e\rangle|^2$.
\end{Thm}

\begin{Lemma}\label{Res=lim}
Assume we have  $\lambda\in(0,1)$, and  $a:\bn\times[0,\infty)\to\br$ continuous in the second variable, such that $\displaystyle\lim_{n\to\infty,\eps\to0}a(n,\eps)=a_0$. Then
$$
\lim_{\eps\to0^+}\eps
\sum_{n\in\bn} \lambda^{n\eps}a(n,\eps)=\frac{a_0}{\log(1/\lambda)}.
$$
\end{Lemma}
\begin{proof}
From the hypothesis, for all $\g>0$, there exist $n_0\in\bn$, $\eps_0>0$, such that $0<\eps<\eps_0$, $n>n_0\Rightarrow |a(n,\eps)-a_0|<\g$. Therefore,
\begin{align*}
\eps\sum_{n\in\bn} \lambda^{n\eps}a(n,\eps)
&< \eps \sum_{n\leq n_0} \lambda^{n\eps}a(n,\eps) + \eps \sum_{n>n_0} \lambda^{n\eps}(a_0+\g)\\
&<\eps\sum_{n\leq n_0} a(n,\eps)+\eps\frac{(a_0+\g) \lambda^{(n_0+1)\eps}}{1-\lambda^{\eps}}
%\\&
\stackrel{\eps\to0}{\to}\frac{a_0+\g}{\log(1/\lambda)}.
\end{align*}
With similar estimates, we get 
$$
\lim_{\eps\to0^+}\eps
\sum_{n\in\bn} \lambda^{n\eps}a(n,\eps)=\frac{a_0-\gamma}{\log(1/\lambda)}.
$$
The thesis follows by the arbitrariness of $\gamma$.
\end{proof}

\begin{proofof}{Theorem \ref{thm:energy}}
Since 
%$\displaystyle [D,f]=\bigoplus_{n\in\bn}\bigoplus_{e\in E_n}\frac1{\ell(e)}\langle \partial f,e\rangle\begin{pmatrix}0&1\\1&0\end{pmatrix}$, we obtain
$\displaystyle
 |[D,f]|^2|D|^{-s}=\bigoplus_{e\in E}\ell(e)^{s-2}|\langle \partial f,e\rangle|^2I_2$, 
\begin{align*}
\tr\big( |[D,f]|^2|D|^{-s}\big)
&=\sum_{n\in\bn}\sum_{e\in E_0}\sum_{|\s|=n}\ell(w_\sigma (e))^{s-2} |\langle \partial f,w_\sigma (e)\rangle|^2\\
&=\sum_{n\in\bn} e^{n((s-2)\log\lambda+\log\rho)}\sum_{e\in E_0}\ell(e)^{s-2}\rho^{-n}\sum_{|\s|=n} |\langle \partial f,w_\sigma (e)\rangle|^2.
%&=\sum_{n\in\bn} \lambda^{n\eps}\rho^{-n}S_n(\ce,\eps)[f].
\end{align*}
We observe that, by Theorem \ref{FS06}, the sequence
$$
\rho^{-n}\sum_{e\in E_0}\sum_{|\s|=n} |\langle \partial f,w_\sigma (e)\rangle|^2.
$$
converges, for $f\in\mathscr{F}$, to a suitable energy form, when $n\to\infty$, hence is bounded from above and from below by suitable constants $M_f$ and $m_f$.
Then, setting $\ell_{max}=\max_{e\in E_0}\ell(e)$, $\ell_{min}=\min_{e\in E_0}\ell(e)$, for $s<2$ we get
$$
m_f\ell_{max}^{s-2}\sum_{n\in\bn} e^{n((s-2)\log\lambda+\log\rho)}
\leq
\tr\big( |[D,f]|^2|D|^{-s}\big)
\leq
M_f\ell_{min}^{s-2}\sum_{n\in\bn} e^{n((s-2)\log\lambda+\log\rho)}.
$$
Therefore the series above converges {\it iff} $s>\delta:=2-\frac{\log \rho}{\log\lambda}$, which proves the first statement.

Let us now consider the energy functional $\ce$ on $C(V_0)$ in the statement. Then 
$\displaystyle
S_n(\ce)[f]
=\sum_{|\s|=n}\sum_{e\in E_0}\ell(e)^{\d-2} |\langle \partial f, w_\s(e)\rangle|^2$,
and $\rho^{-n}S_n(\ce)[f]$ converges to a self-similar Dirichlet form $\ce_\infty[f]$ by Theorem \ref{FS06}. 
Setting 
$\displaystyle
S_n(\ce,\eps)[f]
=\sum_{|\s|=n}\sum_{e\in E_0}\ell(e)^{\delta+\eps-2} |\langle f,\s(e)\rangle|^2
$, we get 
\begin{align*}
\Res_{s=\delta}&\tr (|[D,f]|^2\, |D|^{-s})
=\lim_{\eps\to0^+}\eps
\sum_{n\in\bn} \lambda^{n\eps}\rho^{-n}S_n(\ce,\eps)[f].
\end{align*}
Since $\ell_{min}^\eps S_n(\ce)[f]\leq S_n(\ce,\eps)[f]\leq \ell_{max}^\eps S_n(\ce)[f],$
 the function $(n,\eps)\mapsto \rho^{-n}S_n(\ce,\eps)[f]$ satisfies the hypothesis of Lemma \ref{Res=lim}, so that, when $f$ has finite energy,
$$
\Res_{s=\delta}\tr (|[D,f]|^2\, |D|^{-s})=\frac1{\log1/\lambda}\ce_\infty[f].
$$
\end{proofof}

\subsection{An example of non uniqueness}
When the fractal has a unique self-similar energy form, the Theorem above shows that such unique form can be obtained as a suitable residue. We now discuss the case of the Vicsek fractal, where uniqueness does not hold.

For the Vicsek snowflake, autoforms are parametrized, (up to a scalar multiple), by the conductances $(1,1,1,1,H,H^{-1})$, $H>0$, \cite{Metz}. The corresponding self-similar energies can be recovered with our approach via metric deformations,
% figura
 \begin{figure}[ht]
    \centering
	\includegraphics[width=1.8in]{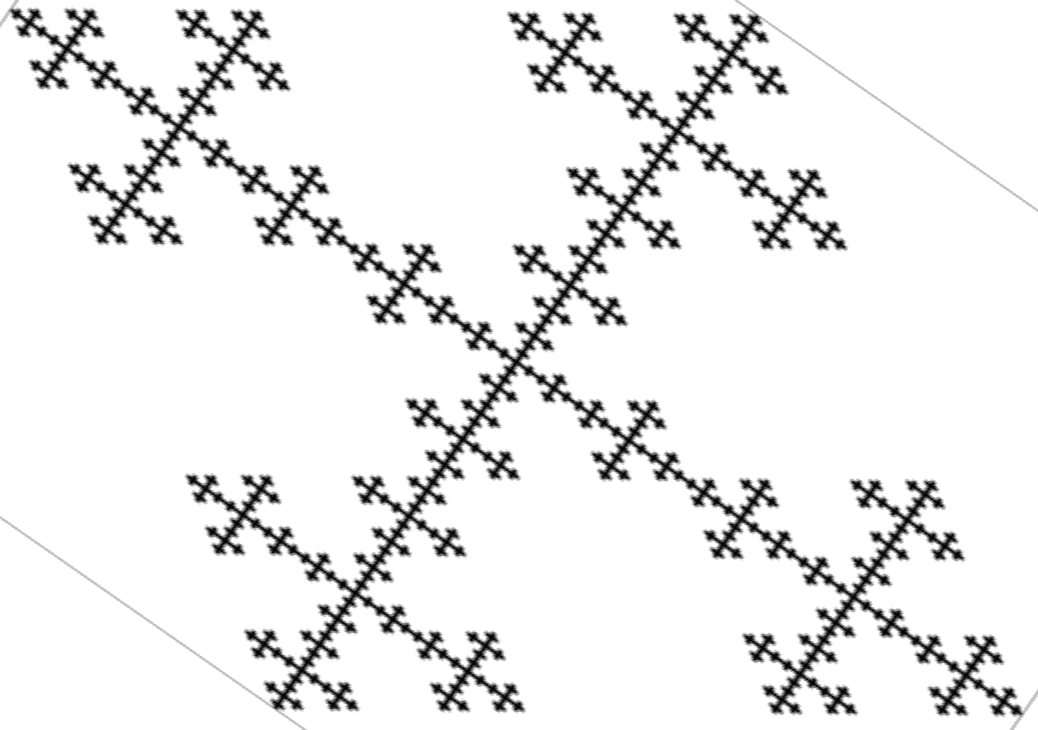} 
    \caption{Rhombic Vicsek snowflake}
    \label{fig:DeformedVicsek}
 \end{figure}
namely are associated with the fractal determined by 5 similitudes with scaling parameter 1/3, whose fixed points coincide with the 4 vertices of a rhombus and with the center of the rhombus itself. 
This means that with our approach the 1-parameter family of energies for the Vicsek square correspond to a 1-parameter family of metrically different fractals.
Let us remark that such deformed Vicsek do not satisfy the  symmetry part of the definition of nested fractal (cf. Definition \ref{def:nested} (5)), however our approach via spectral triples is anyway applicable (see Remark \ref{notnested}), in particular the Zeta function $Z_{|[D,f]|^2}(s)=\tr (|[D,f]|^2\, |D|^{-s})$ still makes sense. The existence of the limit involved in the residue formula and its relation with the energy form will be proven below.

We assume the side of the rhombus has length 1, and angle is $2\th\leq\pi$, so that the diagonals measure $2\sin\th$ and $2\cos\th$, and the ratio between the lengths of the diagonals is $\tan\th$.

\begin{Thm}
Let $K$ be the rhombic Vicsek with angle $2\th$ described above, with the spectral triple as in Definition \ref{DiscreteTriples}. Then the residue of the Zeta function $Z_{|[D,f]|^2}(s)=\tr (|[D,f]|^2\, |D|^{-s})$ at $\delta=1$  exists and
coincides (up to a multiple) with the self-similar energy associated with the eigenform with conductances $(1,1,1,1,H,H^{-1})$ on the Vicsek square, where $\displaystyle H=\frac{2+\sqrt{1+\tan^{2}\th}}{2+\sqrt{1+\cot^2 \th}}$. Let us note that any $H>0$ can be uniquely obtained from a $\th\in(0,\pi)$.
\end{Thm}
\begin{proof}
We observe that the proof of theorem \ref{thm:energy} applies verbatim. Indeed the Zeta function $Z_{|[D,f]|^2}(s)=\tr (|[D,f]|^2\, |D|^{-s})$ depends on the lengths of the edges of the deformed Vicsek, but these are only used to set up a functional $\ce$ on $C(V_0)$ given by $\ce[f]=\sum_{e\in E_0}\ell(e)^{-1} |\langle f,e\rangle|^2$, where
the constants are: $\ell(e)^{-1}=a=1$ for the sides of the rhombus, $\ell(e)^{-1}=g=(2\sin\th)^{-1}=\frac12 \sqrt{1+\cot^{2}\th}$ for the longer diagonal, 
$\ell(e)^{-1}=f=(2\cos\th)^{-1}=\frac12 \sqrt{1+\tan^2\th}$ for the shorter one. Such a functional may then be interpreted as a functional $\ce$ on $C(V_0)$ for the undeformed square. 
Since for the Vicsek square we have $\lambda=1/3$, $\rho=1/3$, we obtain 
that $\delta =1$ and the residue in $\delta=1$ for the Zeta function exists and is a multiple of an energy form on the  Vicsek square.

%For the Vicsek we have $\lambda=1/3$, $\rho=1/3$. Therefore the functional $\ce$ on $C(V_0)$ is given by $
%\ce[f]=\sum_{e\in E_0}\ell(e)^{-1} |\langle f,e\rangle|^2$, where
%the constants are: $\ell(e)^{-1}=a=1$ for the sides of the rhombus, $\ell(e)^{-1}=g=(2\sin\th)^{-1}=\frac12 \sqrt{1+\cot^{2}\th}$ for the longer diagonal, 
%$\ell(e)^{-1}=f=(2\cos\th)^{-1}=\frac12 \sqrt{1+\tan^2\th}$ for the shorter one.

We then make use of  a computation of De Cesaris \cite{DeCesaris}, establishing a relation between   the constants $(a,a,a,a,f,g)$ for a functional on $C(V_0)$ and the conductances $(A,A,A,A,F,G)$, with $FG=A^2$, for the eigenform giving rise to the same energy:
$$
A=\frac{(a+f)(a+g)}{2a+f+g},\quad  F=\frac{(a+f)^2}{2a+f+g},\quad G=\frac{(a+g)^2}{2a+f+g}.
$$
Therefore the normalized $F$ is
$$
H=\frac{F}{A}=\frac{a+g}{a+f}=\frac{1+\frac12 \sqrt{1+\tan^2\th}}{1+\frac12 \sqrt{1+\cot^{2}\th}}
=\frac{2+\sqrt{1+\tan^2\th}}{2+\sqrt{1+\cot^{2}\th}}.
$$
\end{proof}

\section{On the recovery of the geodesic distance induced by the Euclidean structure}
As mentioned in the Introduction, for a given spectral triple $(\ca,\ch,D)$, the (possibly infinite) distance between states on the C$^*$-algebra $\overline\ca$ is given by \cite{ConnesBook}
$$
d_D(\f,\psi)=\sup\{|\f(a)-\psi(a)|:a\in \ca, \|[D,a]\|\leq1\}.
$$

\begin{Dfn}[An essential Lip-norm for spectral triples]
Let us consider the quotient map $p:\cb(\ch)\to\cb(\ch)/\ck$, namely to the Calkin algebra. Then, given a spectral triple $\ct:=(\ca,\ch,D)$, we consider the seminorm
\begin{equation}\label{QuoLipNorm}
\Less(a):=\|p([D,a])\|,\qquad a\in\ca.
\end{equation}
Replacing the seminorm $\|[D,a]\|$ with $\Less(a)$, we get a (possibly infinite) distance between states on $\overline\ca$
\end{Dfn}

We now restrict our attention to the case when $\ca$ is an abelian algebra. Then both seminorms give rise to (possibly infinite) distances on the  compact Hausdorff space $K$ given by the spectrum of the unital C$^*$-algebra generated by $\ca$ in $\cb(\ch)$ according to the following formulas:
\begin{align}
d_D(x,y)&=\sup\{|f(x)-f(y)|:f\in \ca, \|[D,f]\|\leq1\},\label{NCdistance}\\
\dess(x,y)&=\sup\{|f(x)-f(y)|:f\in\ca,\, \Less(f)\leq1\}.\label{essentialdistance}
\end{align}

%When  $K$ is a compact Hausdorff space and
% $\ca$ is an algebra of functions on $K$ which are continuous and for which $[D,f]$ is bounded, we may take the states to be delta-functions, thus getting a pseudo-distance  between points:
%$$
%d_D(x,y)=\sup\{|f(x)-f(y)|:f\in \ca, \|[D,f]\|\leq1\}.
%$$
%Let us recall that such pseudo-distances have  two possible drow-backs: they may not be a distance, namely the pseudo-distance between two distinct points may be zero, and, even if it is, it may induce a topology different from the original one, cf. \cite{Rief} for a complete analysis of these points. 
%In the following, we modify the seminorm $a\mapsto\|[D,a]\|$, thus obtaining a different noncommutative (pseudo-)distance.
%
%
%When  $K$ is a compact Hausdorff space and
% $\ca$ is an algebra of functions on $K$ as above, the seminorm $\Less$ induces a pseudo-distance on $K$:
%\begin{equation}\label{essentialdistance}
%\dess(x,y)=\sup\{|f(x)-f(y)|:f\in\ca,\, \Less(f)\leq1\},\quad x,y\in K.
%\end{equation}

%\begin{Dfn}
We recall that a metric space $K$ is called {\it finitely arcwise connected} \cite{K} if any pair of points can be joined by a rectifiable curve.
Set $\dg(x,y):=\inf\{\ell(\gamma)|\g:[0,1]\to K$ is rectifiable and $\g(0)=x,\g(1)=y\}$.
If $K\subset\br^n$ is finitely arcwise connected, we call such a distance the Euclidean geodesic distance on $K$.
%\end{Dfn}

\begin{Thm}\label{thm:modLip}
Let $K$ be a finitely arcwise connected nested fractal.
The  (possibly infinite) distance induced by $\Less$ on $K$ is indeed finite, and coincides with the Euclidean geodesic distance on $K$.
\end{Thm}

\begin{rem}
$(a)$ Let us observe that, for the spectral triples usually associated with (possibily noncommutative) manifolds, the spectrum of $[D,a]$ has no non-essential parts, hence the seminorm $L_{\text{ess}}$ coincides with the usual seminorm $\|[D,a]\|$. 
\\
$(b)$ 
%This essential seminorm is indeed a Lip-norm, namely vanishes on constant functions only, and induces a distance compatible with the original topology \cite{Rieffel99}.
%\\
%$(c)$ 
When $K$ is a nested fractal, the seminorm $\|[D,a]\|$ produces a distance which is intermediate between the Euclidean distance and the Euclidean geodesic distance on $K$. For example, such distance takes value 1 between the adjacent vertices of a Vicsek snowflake of side 1 instead of the value $\sqrt2$ given by the Euclidean geodesic distance.
\end{rem}

The proof of the Theorem will require some steps.

\subsection{Small triples}

Given the graph $(V_n, E_n)$, let $\ct_n:=\bigoplus_{|\s|=n}\ct_\s$ be the triple  on $\cc(V_n)$ as in Definition \ref{DiscreteTriples}, and the distance $d_n$ on $V_n$  considered as a path space, where paths consist of finite unions of consecutive edges in $E_n$, with length the sum of the lengths of the edges.

\begin{Lemma}\label{LipNorm}
The Lipschitz seminorm 
$$
L_n(f)=\sup_{x\ne y\in V_n}\frac{|f(x)-f(y)|}{d_n(x,y)}
$$ 
induced by $d_n$ coincides with the seminorm $L_{D_n}$ associated with the triple  $\ct_n$. As a consequence, the noncommutative distance induced by $L_{D_n}$ coincides with $d_n$.
\end{Lemma}
\begin{proof} 
By definition, 
$L_{D_n}(f)=\|[D_n,f]\|=\displaystyle\max_{e\in E_n}\frac{|\partial f(e)|}{\ell(e)}$.
Since for any $e\in E_{n}$ $d_{n}(e^+,e^-)=\ell(e)$, $L_{D_n}(f)\leq L_{n}(f)$. Conversely, given $x,y\in V_n$, let $\g=(e_1,\dots e_k)$ be a geodesic path connecting $x$ and $y$, with $e_j\in E_{n}$, $j=1,\dots,k$, so that $d_{n}(x,y)=\displaystyle\sum_{j=1}^k\ell(e_j)$. Then
\begin{align*}
\frac{|f(x)-f(y)|}{d_{n}(x,y)}
&\leq\sum_{j=1}^k\frac{|\langle\partial f,e_j\rangle|}{d_{n}(x,y)}
\leq\max_{j=1,\dots,k}\frac{|\langle\partial f,e_j\rangle|}{\ell(e_j)}\sum_{j=1}^k\frac{\ell(e_j)}{d_{n}(x,y)}
%\\&\leq \sup_{e\in E_{n}}\frac{|\langle\partial f,e\rangle|}{\ell(e)}=
\leq L_{D_n}(f).
\end{align*}
It is  known that the distance on points induced by $L_n$ as in \eqref{essentialdistance} coincides with $d_n$, from which the last statement follows.
\end{proof}

Let us denote by $\cp(V_n,E_n)$ the set of finite paths in $(V_n,E_n)$, the length of a path being induced by the distance $d_n$. The subsets of simple paths in $\cp(V_n,E_n)$, \textit{i.e.} those that visit any vertex at most once, is denoted $\cp_S(V_n,E_n)$. Any simple path  is determined by the finite sequence of its vertices, viceversa any  sequence $(x_0,\dots,x_k)$ of pairwise distinct vertices in $V_n$ gives rise to a simple path if $(x_{j-1},x_j) $ belongs to $E_n$.

\begin{Lemma}\label{discr-curve}
Let $x,y\in V_n$,  $\g:[0,1]\to K$ a simple curve joining $x$ with $y$.
Then there exists a finite sequence $(t_0,\dots,t_q)$ of points in $[0,1]$ such that
\begin{enumerate}
\item $t_{j-1}<t_j$, $j=1,\dots,q$;
\item $\g(t_j)\in V_n$, $j=0,\dots,q$, $\g(t_0)=x$, $\g(t_j)\ne y$, $j=0,\dots,q-1$, $\g(t_i)\ne\g(t_j), i\ne j\in\{0,\dots,q\}$;
\item $\{\g(t),t_{j-1}<t<t_j\}\cap V_n=\emptyset$, $j=1,\dots,q$;
\item $\g(t_q)=y$.
\end{enumerate}
The sequence $(x_0=\g(t_0),\dots,x_q=\g(t_q))$ determines a simple path in $\cp(V_n,E_n)$ which will be denoted by $\Delta_n(\g)$. If $\g$ is rectifiable, $\ell(\Delta_n(\g))\leq\ell(\g)$. 
\end{Lemma}
\begin{proof}
The sequence $(t_0,\dots,t_q)$ will be constructed by induction.\\
Base of the induction: set $t_0=0$.
\\
Step of the induction: given $(t_0,\dots,t_p)$ satisfying $(1), (2), (3)$, either $(4)$ is satisfied, hence the induction stops, or $(4)$ is not satisfied, in which case we set $t_{p+1}=\inf\Omega_p$, with
$$
\Omega_p=\big\{t\in[0,1]:\g(t)\in V_n\setminus\{\g(t_0),\dots,\g(t_p)\}\big\},
$$
and observe that $\O_p$ is not empty since $1\in\O_p$ and $t_{p+1}$ is a minimum since $\O_p$ is compact.
Note that $t_{p+1}\not\in\{t_0,\dots, t_p\}$ by construction and $t_{p+1}\not\in\cup_{j=1,\dots,p}(t_{j-1},t_j)$ by $(3)$, hence $t_{p+1}>t_p$, namely $(t_0,\dots,t_{p+1})$ satisfies $(1)$. It obviously satisfies $(2)$, and satisfies $(3)$ by the minimality of $t_{p+1}$.

We now observe that, by property $(3)$, $G_j=\{\g(t):t_{j-1}<t<t_j\}$ does not intersect $V_n$, hence it is contained in $\cup_{|\sigma|=m}\dot{C}_\s$. Open set condition and connectedness of $G_j$ imply that there exists a single cell $C$ of level $n$ such that $G_j\subset\dot{C}$,  hence $x_{j-1}=\g(t_{j-1})$ and $x_j=\g(t_j)\in C$. Since the graph of a cell is complete, $(x_{j-1},x_j)\in E_n$, therefore $(x_0,\dots,x_q)$ determines a path in $\cp(V_n,E_n)$. The simplicity of $\Delta_n(\g)$ and the inequality are obvious.
\end{proof}

\begin{Lemma}\label{dn<dn+1}
Let $x,y\in V_n$.
\begin{enumerate}
\item $d_n(x,y)\leq d_{n+1}(x,y)$.
\item If, for all $n\in\bn$,  any edge of level $n$ is the union of edges of level $n+1$, the previous inequality is indeed an equality.
\item $d_n(x,y)\leq\dg(x,y)$.
\end{enumerate}
\end{Lemma}
\begin{proof}
$(1)$ Let $\g$ be a geodesic path in $\cp(V_{n+1},E_{n+1})$ connecting $x$ with $y$. By the preceding Lemma,
$$
d_{n}(x,y)\leq\ell(\Delta_n(\g))\leq\ell(\g)=d_{n+1}(x,y).
$$
\\
$(2)$ Given a geodesic path $\g'\in\cp(V_{n},E_{n})$ connecting $x$ with $y$, we may replace any of its edges by the edges of level $n+1$ that cover it. In this way we get a path $\g\in \cp(V_{n+1},E_{n+1})$ with the same length.
\\
$(3)$ Let  $\g$ be a rectifiable curve in $K$ connecting $x$ with $y$.
We have $d_{n}(x,y)\leq\ell(\Delta_n(\g))\leq\ell(\g)$. Taking the infimum over all rectifiable $\gamma$'s we get the thesis.
\end{proof}

\begin{Lemma}\label{dense-rect}
Let $\g:[0,1]\to\br^n$ be a curve, $G$ a dense subset in $[0,1]$, $\ell_G(\g)$ given by
$$
\ell_G(\g)=\sup\{\sum_{j=1}^n |\g(t_j)-\g(t_{j-1} )| : t_0<t_1<\dots<t_n\in G, n\in\bn\}.
$$
Then $\ell_G(\g)=\ell(\g)$, in particular $\g$ is rectifiable {\it iff} $\ell_G(\g)<\infty$
\end{Lemma}
\begin{proof}
Clearly $\ell_G(\g)\leq\ell(\g)$.
Given $t_0<t_1<\dots<t_n\in [0,1]$, choose $s_0<s_1<\dots<s_n\subset G$ such that $| \g(t_j)-\g(s_j)|<\eps/n$, $j=0,\dots n$. Then
\begin{align*}
 \sum_{j=1}^n |\g(t_j)-\g(t_{j-1} )|
 \leq & \sum_{j=1}^n \Big( |\g(s_j)-\g(s_{j-1} )| + |\g(t_j)-\g(s_j)| + |\g(t_{j-1} )- \g(s_{j-1} )| \Big) \\
 \leq&\ell_G(\g)+2\eps,
\end{align*}
hence $\ell(\g)\leq \ell_G(\g)+2\eps$.
The thesis follows by the arbitrariness of $\eps$.
\end{proof}

\begin{Lemma}\label{limitcurve}
Let $x,y$ be vertices in $V_m$, $\{g_n\}_{n\geq m}$ a sequence of simple paths joining $x$ with $y$, $g_n\in \cp_S(V_n,E_n)$, such that $\Delta_k(g_n)=g_k$, for $m\leq k\leq n$, and assume $\ell(g_n)$ is bounded. Then there exists a rectifiable curve $\g$ in $K$ such that $\Delta_n(\g)=g_n$, $n\geq m$, and $\ell(\g)=\lim_n\ell(g_n)$.
\end{Lemma}
\begin{proof}
Let us denote by $V_n(g)$ the sequence of vertices in $V_n$ determined by $g_n$, $n\geq m$. Each $V_n(g)$ is a totally ordered set, and the family of sets $V_n(g)$ is increasing in such a way that the order is preserved, so $V_\infty(g):=\cup_n V_n(g)$ is also a totally ordered set. 
For $m\leq k\leq n$, $u$ preceding $v$ in $V_k(g)$, let us denote by $T_u^v(g_n)$ the sub-path of $g_n$ from $u$ to $v$. Let us note that, for any $u,v$, the sequence $n\to\ell\big(T_u^v(g_n)\big)$ is increasing and bounded.
We now set $R(v)=\lim_n\ell\big(T_x^v(g_n)\big)$, $v\in V_\infty(g)$, denote by $G$ its range, set $c=R(y)$, and observe that $R$ is order-preserving and injective. Moreover, $G$ is dense in $[0,c]$: if not, let $0<a<b<c$ be such that $(a,b)$ is an open interval in $[0,c]\setminus G$, and, for any $n\geq m$, let $u_n$ be the last vertex in $V_n(g)$ such that $R(u_n)\leq a$,  $v_n$ the vertex following $u_n$ in $V_n(g)$, $e_n\in E_n$ the edge joining $u_n$ with $v_n$. Since $R(v_n)>a$, and $R(v_n)\not\in(a,b)$ by hypothesis, $R(v_n)\geq b$. Then
$$
\ell\big((g_n)_x^{y}\big)=\ell\big(T_x^{u_n}(g_n)\big)+\ell(e_n)+\ell\big(T_{v_n}^y(g_n)\big)\leq a+\ell_{max} \l^n+(c-b),
$$
where $\ell_{max}$ denotes the maximum length of an edge in $E_0$. Passing to the limit on $n$ we get $c\leq a+c-b$, namely $a\geq b$, against the hypotheses. Now set $\g:G\to K$, $\g(t)=R^{-1}(t)$, and observe that, given $s<t\in G$, $n$ such that $\g(t),\g(s)\in V_n(g)$, 
$$
|\g(t)-\g(s)|\leq\ell\big(T_{\g(s)}^{\g(t)}(g_n)   \big) \leq t-s,
$$
namely $\g$ is Lipschitz on $G$, therefore it extends to a continuous function on $[0,c]$ with values in $K$. By construction, $\Delta_n(\g)=g_n$, $n\geq m$. By Lemma \ref{dense-rect}, $\g$ is rectifiable. Finally, let $\cx$ be the family of finite subsets of $G$, ordered by inclusion, and, for any $X=\{t_0<t_1<\ldots<t_n\}\in\cx$, set $\ell_X(\g):=\sum_{j=1}^n |\g(t_j)-\g(t_{j-1}|$. Since $X \mapsto \ell_X(\g)$ is increasing, $\lim_{X\in\cx} \ell_X(\g) = \ell_G(\g)$. For any $n\geq m$, let $G_n:=R(V_n(g))$, so that $G_n\in\cx$, and $\ell_{G_n}(\g) = \ell(g_n)$. Since $\{ G_n \}_{n\geq m}$ is cofinal in $\cx$, we get $\lim_n\ell(g_n) = \ell_G(\g) = \ell(\g)$.
%\medskip
%
%If $V_m(g)$ consists of $p+1$ vertices, we associate with each of them a value $j/p$, $j=0,\dots,p$ in an order preserving way. If $y$ and $y'$ are consecutive vertices in $g_m$, associated with $Q(y),Q(y')\in\bq$ respectively, and the sub-path of $g_{m+1}$ from $y$ to $y'$ consists of $r+1$ points, end-points included, we associate with each of them a value $Q(y)+j/r(Q(y')-Q(y))$, $j=0,\dots, r$, in an order preserving way. Continuing by induction we get the mentioned map, and observe that the range $R_Q$ of such map is densely contained in $[0,1]$. We then set $\g(t)=Q^{-1}(t)$, for $t\in R_Q$. Such map is indeed uniformly continuous, hence extends uniquely to a continuous map $\g:[0,1]\to\overline{V_\infty(g)}\subset K$. The property $\Delta_n(\g)=g_n$ follows easily by Lemma \ref{discr-curve}. From Lemma \ref{dn<dn+1}, the sequence $\ell(g_n)$ is increasing and bounded by $\ell(\g)$; the equality $\ell(\g)=\lim_n\ell(g_n)$ then follows by Lemma \ref{dense-rect}.
\end{proof}

\begin{Thm}\label{dgeo=limdn}
For any  $x,y\in V$, $\lim_n d_n(x,y)=\dg(x,y)$. Moreover, there exists a geodesic curve joining $x$ with $y$.
\end{Thm}
\begin{proof}
Assume $x,y\in V_m$. For any $n\geq m$ choose a geodesic path $\g_n\in\cp(V_n,E_n)$ connecting $x$ with $y$; being geodesic, $\g_n$ is simple. 

We now construct by induction sequences $g_n\in\cp_S(V_n,E_n)$, $\O_n\subset\bn$, $n\geq m$ such that
\begin{enumerate}
\item $\Delta_k(g_n)=g_k, m\leq k\leq n$,
\item $\O_n$ is infinite, contained in $[n,\infty)$, and decreasing w.r.t. the inclusion, and for any $p\in\O_n$, $\Delta_n(\g_p)=g_n$.
\end{enumerate} 
Base case: consider $\{\Delta_m(\g_n),n\geq m\}\subset\cp_S(V_m,E_m)$. Since the latter is a finite set, there exists $g_m\in\cp_S(V_m,E_m)$ such that the set $\O_m:=\{n\geq m:\Delta_m(\g_n)=g_m\}$ is infinite.
\\
Inductive step: if we have $g_k\in\cp_S(V_k,E_k)$, $\O_k\subset\bn$, $m\leq k\leq n$, satisfying $(1)$ and $(2)$, construct $g_{n+1}$ and $\O_{n+1}$ as follows.
Consider $\{\Delta_{n+1}(\g_p) : p\in\O_n,p\geq n+1\}\subset \cp_S(V_{n+1},E_{n+1})$. Since the latter is a finite set, there exists $g_{n+1}\in\cp_S(V_{n+1},E_{n+1})$ such that the set $\O_{n+1}:=\{p\in\O_n:p\geq n+1, \Delta_{n+1}(\g_p)=g_{n+1}\}$ is infinite.
Clearly, for $m\leq k\leq n+1$, $p\in\O_{n+1}\subset\O_k$,
$\Delta_k(g_{n+1}) = \Delta_k(\Delta_{n+1}(\g_p)) = \Delta_k(\g_p)=g_k$.

Let us now observe that, setting $n_k=\min\O_k$, the sequence $n_k$ is increasing and tends to $\infty$, and $\ell(g_k)=\ell(\Delta_k(\g_{n_k}))\leq\ell(\g_{n_k})$. 
Moreover, by Lemma \ref{dn<dn+1} the sequence $d_n(x,y)=\ell(\g_n)$ is increasing and bounded, hence has a finite limit, and, by Lemma \ref{limitcurve}, the sequence $g_n$ gives rise to a rectifiable curve $\g$ joining $x$ with $y$.
As a consequence,
$$
\ell(\g)=\lim_k\ell(g_k)\leq\lim_k\ell(\g_{n_k})=\lim_kd_{n_k}(x,y)=\lim_nd_n(x,y),
$$
hence $\dg(x,y)\leq\ell(\g)\leq\lim_nd_n(x,y)\leq\dg(x,y)$, namely $\dg(x,y)=\lim_nd_n(x,y)$ and $\g$ is geodesic.
\end{proof}

\subsection{Intermediate triples}

\begin{Dfn}
Let us set $E_{n,\infty}=\cup_{k\geq n}E_k$, and observe that the graph $(V, E_{n,\infty})$ is connected for any $n\in\bn$. Then, for $x,y\in V$, we pose 
$$
d_{n,\infty}(x,y)=\inf\{\ell(\g) : \g\text{ finite path in }E_{n,\infty} \text{ connecting } x \text{ with } y\}
$$
and consider on $V$ the triple $\ct_{n,\infty}=\oplus_{|\s|\geq n}\ct_\s$, with the  seminorm 
$$
L_{D_{n,\infty}}(f)=\sup_{e\in E_{n,\infty}}\frac{|\langle\partial f,e\rangle|}{\ell(e)}.
$$
Let us observe that $L_{0,\infty}=L_D$ is the seminorm used by Connes, namely it induces the Connes' distance on $K$.

Given $x,y\in V$, define  $k(x,y)=\min\{j\in \bn:x,y\in V_j\}$.
\end{Dfn}

\begin{Lemma}\label{dninfinity=dn}
For $x,y\in V$, $n\geq k(x,y)$,
$d_{n,\infty}(x,y)=d_n(x,y)$. In particular, any two vertices $x,y\in V$ are connected by a geodesic path in $\cp(V,E_{n,\infty})$.
\end{Lemma}
\begin{proof}
Recall that, since $x,y\in V_n$,
$$
d_{n}(x,y)=\inf\{\ell(\g) : \g\text{ finite path in }E_{n} \text{ connecting } x \text{ with } y\},
$$
therefore $d_{n,\infty}(x,y)\leq d_{n}(x,y)$. Conversely, given a path $\g\in\cp(V,E_{n,\infty})$, the path $\Delta_n(\g)$ belongs to $\cp(V_n,E_n)$ and is shorter then $\g$. Taking the infimum over  $\g$ we get the first statement. As for the second, if $k(x,y)\leq n$ take the geodesic path in $(V_n,E_n)$. If $k=k(x,y)>n$, take the geodesic path $\gamma \in \cp(V_k,E_{n,k})$, where $E_{n,k}=\cup_{n\leq p\leq k}E_p$. Given any other path $\gamma'\in \cp(V,E_{n,\infty})$ we may consider the path $\Delta_k(\g')$, which is shorter and belongs to $\cp(V_k,E_k)\subset\cp(V_k,E_{n,k})$, hence $\ell(\g)\leq\ell(\Delta_k(\g'))\leq\ell(\g')$. The thesis follows.
\end{proof}

Let us now consider the Lipschitz seminorm $L_{n,\infty}$ on functions on $V$ associated with the distance $d_{n,\infty}(x,y)$,
$$
L_{n,\infty}(f)=\sup_{x, y\in V}\frac{|f(x)-f(y)|}{d_{n,\infty}(x,y)}.
$$

\begin{Lemma}\label{Ln=LDn}
$\displaystyle L_{n,\infty}(f)=L_{D_{n,\infty}}(f)$, hence the noncommutative distance induced by $L_{D_{n,\infty}}(f)$ on $V$ coincides with $d_{n,\infty}$. In particular, setting $n=0$, we obtain that the Connes' distance on $V$ coincides with $d_{0,\infty}$.
\end{Lemma}
\begin{proof}
Since, by Lemma \ref{dninfinity=dn}, any two vertices $x,y\in V$ are connected by a geodesic path, the proof of Lemma \ref{LipNorm} applies verbatim.
\end{proof}

\subsection{Conclusion}

Let us now observe that, for any Lipschitz function $f$, the operator $|[D,f]|$ is completely diagonalizable, with eigenvalues 
$$
\Big \{ \frac{|\langle \partial f,e\rangle|}{\ell(e)}:e\in E_{0,\infty} \Big \} = \bigcup_{n\in\bn} \Lambda_n,
\quad \Lambda_n = \Big \{ \frac{|\langle \partial f,e\rangle|}{\ell(e)}:e\in E_n \Big \}.
$$
Therefore the essential spectrum of $|[D,f]|$ is given by the limit points of the sequences
$\big\{\{a_n\}:a_n\in\Lambda_n\big\}$, the norm in the Calkin algebra is given by the maximum of the essential spectrum, hence
$$
L_{\text{ess}}(f)=\|p([D,f])\|=\inf_n\sup_{e\in E_{n,\infty}}\frac{|\langle \partial f,e\rangle|}{\ell(e)}=\inf_n L_{n,\infty}(f).
$$

\begin{Prop}\label{ess=geoOnV}
For any $x,y\in V$, $\dess(x,y)=\dg(x,y)$, and, if $\Less(f)=1$ then $|f(x)-f(y)|\leq\dess(x,y)$.
\end{Prop}
\begin{proof}
Since, for any $x,y\in V$, $d_{n,\infty}(x,y)$ is increasing with $n$,
%Let $n\geq k(x,y)$.
\begin{align*}
\dess(x,y)&=\sup_{f\in\cc(K)}\frac{|f(x)-f(y)|}{\Less(f)}
%\\&
=\sup_{f\in\cc(K)}\sup_n\frac{|f(x)-f(y)|}{L_{n,\infty}(f)}\\
&=\sup_n\sup_f\frac{|f(x)-f(y)|}{L_{n,\infty}(f)}
%\\&
=\sup_nd_{n,\infty}(x,y)\\
&=\lim_nd_{n,\infty}(x,y)=\lim_nd_{n}(x,y)=\dg(x,y).
\end{align*}
We note in passing that the request ${f\in\cc(K)}$ plays no role when $x,y\in V$.
As for the second statement, $\Less(f)=1$ means $L_{n,\infty}(f)=1+\eps_n$, with $\eps_n\to0$. Then, for $n\geq k(x,y)$, $|f(x)-f(y)|\leq(1+\eps_n)d_n(x,y)$. Passing to the limit on $n$ one gets the thesis.
\end{proof}

\begin{proofof}{Theorem \ref{thm:modLip}}
Let $x\ne y\in K$, $\g$ a rectifiable curve joining $x$ with $y$ parametrised by arc-length, and set $c=\ell(\g)$. 
Let us note that the set $I_{\g}=\{t\in[0,c]:\g(t)\in V\}$ is dense in $[0,c]$. If not, we would get an interval $(a,b)\subset[0,c]$ such that, for any $n\in\bn$, $\g_o:=\{\g(t):t\in(a,b) \}\cap V_n=\emptyset$, namely, for any $n\in\bn$, $\g_o$ is contained in a single cell $C_n$ of level $n$, hence $\g_o\subset\cap_n C_n$, which consists of at most one point, that is the interval $(a,b)$ is empty.
Therefore, there exists a minimum $k$ such that the set $I_{\g,p}=\{t\in[0,c]:\g(t)\in V_p\}$,  $p\geq k$, is non-empty and closed, by continuity of $\g$. We set $s_p=\min I_{\g,p}$ and $t_p=\max I_{\g,p}$. 
By the density of $I_{\g}$, $s_p$ decreases to $0$ and $t_p$ increases to $c$, which implies that, setting $x_p=\g(s_p)$ and $y_p=\g(t_p)$, $d(x,x_p)\leq\ell(\g|_{[0,s_p]})=s_p\to0$, and $d(y,y_p)\leq\ell(\g|_{[t_p,c]})=c-t_p\to0$.
As a consequence, if $f\in\cc(K)$ and $\Less(f)\leq1$,
\begin{align*}
|f(x_k)-f(x)|&=\lim_{p\in\bn}|f(x_k)-f(x_p)|=|\sum_{p\geq k}f(x_p)-f(x_{p+1})|\\
&\leq\sum_{p\geq k} |f(x_p)-f(x_{p+1})|\leq \sum_{p\geq k} \dess(x_p,x_{p+1})\\
&\leq\sum_{p\geq k} \ell\big(\g|_{[s_{p+1},s_p]}\big)=\sum_{p\geq k} (s_p-s_{p+1})=s_p,
\end{align*}
where, in the first equality we used the continuity of $f$, in the second inequality the inequality in Proposition \ref{ess=geoOnV}, in the last inequality the identity $\dess=\dg$ on vertices of  Proposition \ref{ess=geoOnV}, and in the last but one equality the fact that $\g$ is parametrized by arc length. Reasoning in the same way, one gets $|f(y_k)-f(y)|\leq c-t_k$, hence
\begin{align*}
|f(x)-f(y)|&\leq |f(x_k)-f(y_k)|+|f(x_k)-f(x)|+|f(y_k)-f(y)|\\
&\leq |f(x_k)-f(y_k)|+s_k+(c-t_k);\\
|f(x_k)-f(y_k)|&\leq |f(x)-f(y)|+|f(x_k)-f(x)|+|f(y_k)-f(y)|\\
&\leq |f(x)-f(y)|+s_k+(c-t_k).
\end{align*}
As a consequence, 
$$
\dess(x,y)=\sup\{|f(x)-f(y)|:f\in\cc(K),\, \Less(f)\leq1\}\leq \dess(x_k,y_k)+s_k+(c-t_k),
$$
and, analogously, $\dess(x_k,y_k)\leq \dess(x,y)+s_k+(c-t_k)$, which implies 
$|\dess(x_k,y_k)-\dess(x,y)|\leq  s_k+(c-t_k)\to0$, $k\to\infty$, namely
$$
\dess(x,y)=\lim_k \dess(x_k,y_k).
$$
On the other hand, 
\begin{align*}
|\dg(x_k,y_k)-\dg(x,y)|&\leq \dg(x_k,x) + \dg(y_k,y) \\
&\leq \ell\big(\g|_{[0,s_k]}\big)+\ell\big(\g|_{[t_k,c]}\big)=s_k+(c-t_k),
\end{align*}
from which
$$
\dg(x,y)=\lim_k \dg(x_k,y_k)=\lim_k \dess(x_k,y_k),
$$
which immediately implies the thesis.
%Since $d_{\text{ess}}=\dg$ on $V$ and $V$ is dense in $K$, we get
%$$
%L_{\text{ess}}(f|_V)=L_{\text{geo}}(f|_V)=L_{\text{geo}}(f)
%$$
%Therefore
%\begin{align*}
%d_{\text{ess}}(x,y)&=\sup\{|f(x)-f(y)|:f\in\cc(K),\ L_{\text{ess}}(f|_V)\leq1\}\\
%&=\sup\{|f(x)-f(y)|: L_{\text{geo}}(f)\leq1\}=\dg(x,y).
%\end{align*}
\end{proofof}

%\noindent
%{\it Osservazione: la proprieta' 2 degli spazi geodetici non e' fondamentale: definendo la distanza geodetica come un inf, e facendo modifiche minime, si puo' comunque dimostrare che $d_{ess}=d_{geo}$. La mia impressione e' che la terza proprieta' sia automatica per frattali autosimili, data la prima.}

\begin{Cor}
Assume that  any edge of level $n$ is the union of edges of level $n+1$. Then the Connes' distance $d_D$ for the triple $\ct$ coincides with the geodesic distance, hence with the essential distance.
\end{Cor}
\begin{proof}
Let $x,y\in V$. Then, by Lemma \ref{dninfinity=dn}, $d_{0,\infty}(x,y)$ is realized by a geodesic path in $\cp(V,E_{0,\infty})$, which, as in the proof of the Lemma, can be taken in $\cp(V,E_{0,k})$, with $k=k(x,y)$. Since any edge of level $n$ is the union of edges of level $n+1$, such path has an identical counterpart (as a set) in $\cp(V,E_k)$, namely $d_{0,\infty}(x,y)=d_k(x,y)$.
Making use of Lemma \ref{dn<dn+1} and of Theorem \ref{dgeo=limdn}, we get $d_{0,\infty}(x,y)=\dg(x,y)$, and finally, by Lemma \ref{Ln=LDn}, $d_D(x,y)=\dg(x,y)$. The step from vertices to general points can be proved as in the Theorem above.
\end{proof}

\begin{rem}
As mentioned at the beginning, our spectral triples are based here on the complete graph with vertices $V_0$. Therefore the hypothesis of the Corollary is satisfied e.g. for the generalized Sierpinski triangles in the plane obtained by contractions of $1/p$, or by the higher-dimensional gaskets inscribed in $n$-simplices. However it is not satisfied for the poly-gaskets ($N>3$) in \cite{BCFRT}, nor for the Lindstr\o m or Vicsek snowflakes. The von Koch curve is not even finitely arcwise connected.
\end{rem}

%\appendix
%
%\section{Results on Lip-norms}
%The following Lemma is well known, we give a proof for completeness.
%\begin{Lemma}\label{ConnesDistance}
%Let $(X,d)$ be a metric space,  set $L(f)=\displaystyle \sup_{x, y\in X}\frac{|f(x)-f(y)|}{d(x,y)}$, and $\tilde{d}(x,y)=\displaystyle\sup\{|f(x)-f(y)|:L(f)\leq1\}$. Then, $d=\tilde{d}$.
%\end{Lemma}

%: REFERENCES


\begin{thebibliography}{99}

\bibitem{Barl} 
M.~T.~Barlow.  
{\it Diffusions on fractals}, 
in ``Lectures on Probability Theory and Statistics'', 
Lecture Notes in  Math. {\bf 1690}, Springer, New York, 1998, pp. 1-121.

\bibitem{PeBe}
J. V. Bellissard, J. C. Pearson. 
{\it Noncommutative Riemannian geometry and diffusion on ultrametric Cantor sets}, 
J. Noncommut. Geom. \textbf{3} (2009), no. 3, 447-480.

\bibitem{BoMu} J.A. Bondy, U.S.R. Murty.  Graph Theory with Applications, North-Holland, 1976.

\bibitem{BCFRT}
B.~Boyle, K.~Cekala, D.~Ferrone, N.~Rifkin, A.~Teplyaev.
{\it Electrical Resistance of N-Gasket Fractal Networks}, 
Pacific Journal of Mathematics, \textbf{233} (2007), 15-40.

\bibitem{Buyalo} 
S.V.Buyalo. 
{\it Measurability of self-similar spectral geometries}, 
Algebra i Analiz, {\bf 12} (2000), no. 3, 1-39; translation in St. Petersburg Math. J., {\bf 12} (2001), 353-377.


\bibitem{CPS}
A.L.~Carey, J.~Phillips, F.~A.~Sukochev. 
{\it Spectral Flow and Dixmier Traces}, 
Advances in Math., {\bf 173} (2003), 68-113.

%\bibitem{CRSS}
%A. L. Carey, A. Rennie, A. Sedaev, F. A. Sukochev.
%{\it The Dixmier trace and asymptotics of zeta functions}, 
%Journal of Functional Analysis, {\bf 249} (2007), 253-283.

\bibitem{ChIv07} 
E.~Christensen, C.~Ivan.
{\it Sums of two-dimensional spectral triples}, 
Math. Scand., {\bf 100} (2007), 35-60.


\bibitem{CIL}
E.~Christensen, C.~Ivan, M.~L.~Lapidus.
{\it Dirac operators and spectral triples for some fractal sets built on curves}.
Adv. Math., {\bf 217} (2008), 42-78.

\bibitem{CIS}
E.~Christensen, C.~Ivan, E.~Schrohe.
{\it Spectral triples and the geometry of fractals},
J. Noncommut. Geom., {\bf 6} (2012), 249-274.

\bibitem{CiSa1} 
F.~Cipriani, J-L.~Sauvageot. 
{\it Derivations as square roots of Dirichlet forms}, 
J. Funct. Anal., \textbf{201} (2003), 78-120.

\bibitem{CiSa2} 
F.~Cipriani, J-L.~Sauvageot. 
{\it Fredholm modules on p.c.f. self-similar fractals and their conformal geometry}, 
Commun. Math. Phys., \textbf{286} (2009), 541-558.

\bibitem{CGIS02} 
F.~Cipriani, D.~Guido, T.~Isola, J-L.~Sauvageot. 
{\it Spectral triples for the Sierpinski Gasket}, 
J. Funct. Anal., \textbf{266} (2014), 4809-4869.

\bibitem{ConnesBook} A. Connes. Noncommutative geometry, Academic Press, 1994.

\bibitem{DeCesaris} J.~De Cesaris, PhD thesis, Univ. Roma Tor Vergata, 2012.

\bibitem{FaSa} K.~Falconer,  T.~Samuel.
{\it Dixmier traces and coarse multifractal analysis}, 
Ergodic Theory Dynam. Systems, \textbf{31} (2011), 369-381.

\bibitem{GuIs8} 
D.~Guido, T.~Isola.
{\it Fractals in noncommutative geometry}, 
in  ``Mathematical physics in mathematics and physics'' (Siena, 2000), 
171-186, Fields Inst. Commun., 30, Amer. Math. Soc., Providence, RI, 2001.


\bibitem{GuIs9} 
D.~Guido, T.~Isola. 
{\it Dimensions and singular traces for spectral triples, with applications to fractals}, 
Journ. Funct. Analysis, (2) \textbf{203}  (2003), 362-400. 

\bibitem{GuIs10}  
D.~Guido, T.~Isola. 
{\it Dimensions and spectral triples for fractals in $\br^n$}, 
in ``Advances in Operator Algebras and Mathematical Physics'', Proceedings of the Conference held in Sinaia, Romania, June 2003, F. Boca, O. Bratteli, R. Longo H. Siedentop Eds., Theta Series in Advanced Mathematics, Bucharest 2005. 

\bibitem{GuIsProc} 
D.~Guido, T.~Isola. 
{\it New results for old spectral triples}.
Noncommutative analysis, operator theory and applications, 261-270, Oper. Theory Adv. Appl., 252, Birkh\"auser/Springer, 2016.

\bibitem{HKT}
M.~Hinz, D.J.~Kelleher, A.~Teplyaev. 
{\it Metric and spectral triples for Dirichlet and resistance forms}, 
Journal of Noncommutative Geometry, {\bf 9} (2015), 359-390.

%\bibitem{HKMRS}
%J. Hyde, D. J. Kelleher, J. Moeller, L. G. Rogers, L. Seda. 
%{\it Magnetic Laplacians of locally exact forms on the Sierpinski Gasket}, 
%arXiv:1604.01340, 2016.

%\bibitem{JKS}
%A.~Julien, J.~Kellendonk, J.~Savinien. 
%{\it On the Noncommutative Geometry of Tilings}, 
%in ``Mathematics of Aperiodic Order'', J. Kellendonk, D. Lenz, J. Savinien Eds, 
%Progress in Mathematics, Vol. 309, 259-309, Springer Basel, 2015.

\bibitem{JuSa}
A.~Julien, J.~Savinien. 
{\it Transverse Laplacians for substitution tilings}, 
Comm. Math. Phys., \textbf{301} (2011),  285-318.

\bibitem{Kiga1}
J. Kigami. 
{\it Harmonic calculus on p.c.f. self-similar sets}, 
Transactions of the American Mathematical Society, {\bf 335} (1993), 721-755.

\bibitem{Kiga2}
J. Kigami. 
{\it Analysis on fractals}, 
Cambridge University Press, 2001.

\bibitem{KiLa1}
J. Kigami, M. L. Lapidus. 
{\it Weyl's problem for the spectral distribution of the Laplacian on p.c.f. self-similar fractals}, 
Communications in Mathematical Physics, {\bf 158} (1993), 93-125.

\bibitem{KiLa2}
J. Kigami, M. L. Lapidus. 
{\it Self-similarity of volume measures for Laplacians on p.c.f. self-similar fractals}, Communications in Mathematical Physics, {\bf 217} (2001), 165-180.

\bibitem{K}  
S.~Kobayashi. 
{\it Invariant distances on complex manifolds and holomorphic mappings}, 
J. Math. Soc. Japan, {\bf 19} (1967), 460-.

\bibitem{Kusu}
S. Kusuoka. 
{\it Dirichlet forms on fractals and products of random matrices}, 
Publications of the Research Institute for Mathematical Sciences, {\bf 25} (1989), 659-680.

\bibitem{La94} 
M.~L.~Lapidus.
{\it Analysis on fractals, Laplacians on self-similar sets, noncommutative geometry and spectral dimensions},
Topol. Methods Nonlinear Anal., {\bf 4} (1994), 137-.

\bibitem{La97} 
M.~L.~Lapidus, 
{\it Towards a noncommutative fractal geometry? Laplacians and volume measures on fractals}, 
in ``Harmonic analysis and nonlinear differential equations'', 211-, 
Contemp. Math., 208, Amer. Math. Soc., Providence, RI, 1997.

\bibitem{LaSa} 
M.~L.~Lapidus, J.~Sarhad.
{\it Dirac operators and geodesic metric on the harmonic Sierpinski gasket and other fractal sets}, 
J. Noncommut. Geom., {\bf 8} (2014), 947-.


\bibitem{Lind} 
T.~Lindstrom. 
{\it  Brownian Motion on Nested Fractals}, 
Mem. Amer. Math. Soc., \textbf{420} (1990).

\bibitem{Metz} 
V.~Metz. 
{\it How Many Diffusions Exist on the Vicsek Snowflake?}, 
Acta Applicandae Mathematicae, \textbf{32} (1993), 227-241.	

\bibitem{Palm}
I. C. Palmer.
{\it Riemannian geometry of compact metric spaces}, 
PhD Dissertation, Georgia Institute of Technology 2010.

\bibitem{Pe00} 
R.~Peirone.  
{\it Convergence and uniqueness problems for Dirichlet forms on fractals}, 
Bollettino U.M.I.,  (8) {\bf 3-B} (2000), 431-460.
 
\bibitem{Pe04} 
R.~Peirone.  
{\it Convergence of discrete Dirichlet forms to continuous Dirichlet forms on fractals}, Potential Analysis,   {\bf 21} (2004), 289-309.

\bibitem{Rieffel99} 
M.~A. Rieffel. 
{\it Metrics on state spaces}, 
Doc.  Math.,  {\bf 4} (1999), 559-600.

\bibitem{Sabot} 
C.~Sabot.  
{\it Existence and uniqueness of diffusions on finitely ramified self-similar fractals}, 
Ann. Sci. \'Ecole Norm. Sup., (4)   {\bf 30} (1997), 605-673.



\end{thebibliography}
\end{document}